\documentclass{amsart}
\usepackage{amsmath,amsfonts,amssymb,amsthm, esint}
\usepackage{mathtools}
\usepackage[english]{babel}
\usepackage{enumerate}
\usepackage{dsfont}
\usepackage{comment}
\usepackage{bm}

\usepackage[colorlinks,citecolor=blue]{hyperref}
\usepackage{epsfig,graphicx}

\newtheorem{theorem}{Theorem}[section]

\newtheorem{lemma}[theorem]{Lemma}
\newtheorem{proposition}[theorem]{Proposition}
\newtheorem{corollary}[theorem]{Corollary}

\theoremstyle{remark}
\newtheorem{remark}[theorem]{Remark}

\numberwithin{equation}{section}
\numberwithin{figure}{section}

\newcommand{\R}{\mathbb{R}}
\newcommand{\N}{\mathbb{N}}

\DeclareMathOperator{\supp}{supp}
\DeclareMathOperator{\dist}{dist}

\DeclareMathOperator{\diam}{diam}
\newcommand{\T}{{\mathrm{T}}}

\newcommand{\one}{\mathbf{1}}
\newcommand{\abs}[1]{\left\lvert#1\right\rvert}
\newcommand{\norm}[1]{\left\lVert#1\right\rVert}
\newcommand*{\eqdef}{\coloneqq}

\DeclareMathOperator{\strain}{strain}

\DeclareMathOperator{\tr}{trace}
\renewcommand{\dist}{\mathsf{d}}
\newcommand{\kap}{\mathsf{k}}
\newcommand{\Kap}{K}

\begin{document}
	\author{Alexey Kroshnin}
 \address{
HSE University 
and Institute for Information Transmission Problems, Moscow
%
}
\email{akroshnin@hse.ru}

	\author{Eugene Stepanov}
	\address{St.Petersburg Branch of the Steklov Mathematical Institute of the Russian Academy of Sciences,
		Fontanka 27,
		191023 St.Petersburg, Russia
		\and
		HSE University, Moscow
			\and
		Department of Mathematical Physics, Faculty of Mathematics and Mechanics,
		St. Petersburg State University, St.Petersburg, Russia	
	}
	\email{stepanov.eugene@gmail.com}
	
	\author{Dario Trevisan}
	\address{Dario Trevisan, Dipartimento di Matematica, Universit\`a di Pisa \\
		Largo Bruno Pontecorvo 5 \\ I-56127, Pisa}
	\email{dario.trevisan@unipi.it}
	
	\thanks{		
	The work of the first and second authors on the paper has been carried out within the framework of the HSE University Basic Research Program.
	The results of Section~\ref{sec:stability} have been obtained under support of the RSF grant \#19-71-30020.
	The third author was partially supported by GNAMPA-INdAM 2020 project ``Problemi di ottimizzazione con vincoli via trasporto ottimo e incertezza'' and University of Pisa, Project PRA 2018-49.
	}
	\date{\today}
	
	\title{Infinite multidimensional scaling for metric measure spaces}
	
	\begin{abstract}
		For a given metric measure space $(X,\dist,\mu)$ we consider finite samples of points, calculate the matrix of distances between them and then reconstruct the points in some finite-dimensional space using the multidimensional scaling (MDS) algorithm
		with this distance matrix as an input. We show that this procedure gives a natural limit as the number of points in the samples grows to infinity and the density of points approaches the measure $\mu$. This limit can be
		viewed as ``infinite MDS'' embedding of the original space, now not anymore into a finite-dimensional space
		but rather into an infinite-dimensional Hilbert space. We further show that this embedding is stable with respect to the natural convergence of metric measure spaces. However, contrary to what is usually believed in applications, we show that
		in many cases it does not preserve distances, nor is even bi-Lipschitz, but may provide snowflake (Assouad-type) embeddings of the original space to a Hilbert space (this is, for instance, the case of a sphere and a flat torus equipped with their geodesic distances).  
	\end{abstract}
	 
	\maketitle
	
\section{Introduction}

Let $(X, \dist)$ be a metric space and suppose that, for a set of points  $\{x_1, \dots, x_n\} \subset X$ the distances $d_{ij} \eqdef \dist(x_i,x_j)$ have been calculated. A common problem in distance geometry is to reconstruct the space $X$, knowing only the information on these distances. Of course, in general, unless $X$ is finite itself one might only hope to do this in the limit as $n\to\infty$, i.e.\ with an infinitesimal error as the number of points becomes
large. One of the classical algorithms aimed to solve this problem is \textit{multidimensional scaling} (MDS), which is widely used in applications, in particular as a background tool in data science for dimension reduction (it is, for instance, a basic part of Isomap dimension reduction method \cite{tenenbaum1997mapping, tenenbaum2000global}), for visualization  and clustering  (see e.g.~\cite{borg2005modern,  wang2012geometric, chesser2020learning}). 

MDS was originally coined for intrinsically Euclidean data, i.e.\ when $X \subset  \R^d$ and the distance $\dist(x,y) = |x-y|$ is Euclidean, where it is strongly linked to Principal Component Analysis (PCA), and a rigorous proof that in this case it in fact reconstructs the distances exactly is well known~\cite{wang2012geometric}. However, most interest for applications relies on its extension to arbitrary distances or even just for general dissimilarity matrices instead of distance matrices, in which case it can be assimilated to a kernel PCA method \cite{bengio2004learning}. Quite recently, it has been shown in~\cite{adams-blumstein-kassab2020MDS} that exact reconstruction of general, not necessarily Euclidean, metric spaces, cannot be achieved by MDS: for $X = \mathbb{S}^1$ a unit radius circle endowed with its geodesic distance instead of the Euclidean one, and $\{ x_i \}_{i=1}^n$ a uniformly spaced grid, MDS yields in the limit as $n \to \infty$ a closed curve, which is very far from being a circle. Infact, a simple computation shows that it is a fractal object, namely, a \emph{snowflake} embedding \cite{tyson2005characterizations} of a circle in an infinite-dimensional Hilbert space~\cite{PuchSpokSteTrev-manif20}, i.e., it becomes an isometric embedding if $S^1$ is endowed with the geodesic distance 
 raised to some power $\alpha \in (0,1)$ (precisely, for  $\alpha=1/2$). Although this may be unexpected in view of various commonly used applications of MDS, an explanation of this fact may be also traced back to the classical work \cite{von1941fourier}, where all the invariant metrics on the circle that embed isometrically into a Hilbert space are classified, including of course the $1/2$-snowflake re-obtained via MDS, which is actually credited to the earlier work~\cite{wilson1935certain} (see also the discussion in~\cite[section 7.3]{kassab2019MDS}).

\subsection{Our contribution}
    
In this paper we show that the above mentioned example is by no means exceptional, namely, in fact MDS for Riemannian manifolds equipped with intrinsic distances quite usually reconstructs snowflake embeddings of such manifolds in an infinite-dimensional Hilbert space. Moreover, we provide natural stability results for MDS maps with respect to their data in terms of natural distances between metric measure spaces. Our results can be summarized as follows:
\begin{enumerate}[i)]
\item We consider the general MDS on metric measure spaces proposed in \cite{kassab2019MDS,adams-blumstein-kassab2020MDS}, which we often refer to as ``infinite MDS'', as opposed to the usual MDS on a finite subset, and show that, under mild assumptions satisfied on any Riemannian manifold with volume measure, the induced maps are ``almost everywhere injective'' in a natural sense (Theorem \ref{theo:d2_decomp} and Corollary \ref{co_MMinj1}), with an explicit inequality.
\item We prove general stability results for the MDS maps in terms of the Gromov-Kantorovich distance of order $4$ (Theorem~\ref{theo:GW_stability}), which in particular shows convergence and identifies the limit of the MDS maps obtained by a sequence of families of points $\{x_i\}_{i=1}^n \subset  X$, in terms of the limit of their empirical measures $\mu_n := \frac{1}{n} \sum_{i=1}^n \delta_{x_i}$, thus settling~\cite[conjecture 6.1]{adams-blumstein-kassab2020MDS}.
\item We prove that, for every $d \ge 1$, in the case of a sphere $\mathbb{S}^d$ (Proposition~\ref{prop_MDSsphere1}) and a flat torus $(\mathbb{S}^1)^d$ (Proposition~\ref{prop_MDStor1} which follows from a more general result on MDS over product spaces, Proposition~\ref{prop_MDSprod1}), endowed with their uniform measures, the MDS maps provide snowflake embeddings in infinite dimensional Hilbert spaces, extending the case of $d=1$ and answering \cite[question 7.1]{adams-blumstein-kassab2020MDS}.
\end{enumerate}
%
%
%
%
%
These results together show that  MDS on metric measure spaces behaves in several aspects quite similarly to other embedding methods used in manifold learning, e.g.\  via eigenfunctions of a Laplace operator or via a heat kernel~\cite{AmbrHonPortTewod-LaplEmbed21,BerBessGall94}, but with strikingly different features, since contrary to these embeddings, it is not even Lipschitz for quite simple manifolds  such as spheres. 

We also point out to the readers the very recent preprint~\cite{lim2022classical} dedicated to the same problem, which appeared on the arXiv around the same time as ours and provides similar results. Although the convergence results formulated there are a bit weaker than ours, one can find there some further curious examples of MDS embeddings.

\subsection{Further problems} 
We believe that our results contribute to the understanding of the structure of embeddings obtained via MDS on metric measure spaces. Still, several  important questions that we describe below remain open.
\begin{enumerate}
\item The most natural question is whether the infinite MDS map is in fact a homeomorphism onto its range, at least for a reasonably large class of metric measure spaces. This reduces to the question whether it is injective (here we are able to prove only a weaker result on ``almost everywhere injectivity''); a partial answer is discussed in Remark~\ref{rem:full-injectivity}, where it is linked to asymptotics of eigenvalues and eigenfunctions of the associated operator, which are however quite far from being easy to obtain even in very explicit examples.
\item The second question concerns the quantitative description of the limit behavior  of (finite) MDS maps in case of randomly sampled points, which is perhaps the most common assumption on the data in applications. Our results may be interpreted in this case as a sort of a qualitative law of large numbers, and more precise information on the order of convergence to infinite MDS maps is quite important to obtain.
\item The third question is to understand the implications of our results and techniques to more sophisticated algorithms, such as Isomap, the stability of which has also been subject of recent research~\cite{balasubramanian2002isomap, arias2020perturbation}, or to the recently proposed method~\cite{PuchSpokSteTrev-manif20} by the last two authors, together with N.~Puchkin and V.~Spokoiny, based on a semidefinite programming problem.
\end{enumerate}

\subsection{Structure of the paper} The paper is structured as follows. In Section~\ref{sec:notation} we recall some general notation and in Section~\ref{sec:MDS} we focus on MDS, both on a finite sample and on a generic metric measure space, showing in particular some basic properties, including continuity of the embedding map. Section~\ref{sec:inject} is dedicated to the proof of its almost everywhere injectivity, Section~\ref{sec:stability} to that of stability with respect to Gromov-Kantovorich convergence. Section~\ref{sec:mds-spheres} deals with the examples of spheres and product spaces, in particular,  flat tori.

\subsection{Acknowledgements} The authors thank two anonymous referees for many constructive comments and in particular for suggesting the use of Vitali spaces and~\cite{heinonen2015sobolev} to dispense from the doubling assumption in Remark~\ref{rem:vitali}.

\section{Notation and preliminaries}\label{sec:notation}

The metric spaces $(X, \dist)$ in the sequel will be always assumed separable and complete, and the measures $\mu$ will be finite and Borel.
For a metric space $(X, \dist)$ and a set $D\subset X$, we will denote by $\bar D$ its closure,
by $\one_D$ its characteristic function.
The notation $B_r(x)\subset X$ stands for the open ball of $X$ centered at $x\in X$ with radius $r>0$.
For a number $x\in \R$ we denote by $x^+$ its positive part, i.e.\ $x^+:=x$ for $x>0$ and $x^+:=0$ otherwise. 

The Euclidean space $\R^m$ is always assumed to be equipped with the Euclidean norm $|\cdot|$ and by $x\cdot y$ we mean the usual Euclidean scalar product of vectors $x$ and $y$. 
If $E$ is a Banach space, then we will denote by $\norm{\cdot}_E$ its norm.
We let $\ell^p$ stand for the usual Banach space of $p$-summable sequences, and $\R^\infty$ stand for the linear
space of all real valued sequences (also denoted by $\R^\N$ in some literature), equipped with its product topology which is metrizable, e.g. by the distance
\[
d(x, y) \eqdef \sum_{n = 1}^\infty 2^{-n} \frac{\abs{x_n - y_n}}{1 + \abs{x_n - y_n}}
\]
which makes it a Polish space.
We will always assume $\R^m\subset \ell^p$ and $\R^m\subset \R^\infty$, the restrictions of
$\ell^p$ and $\R^\infty$ to $\R^m$ being the restrictions to first $m$ coordinates.
If $(X,\dist)$ is a metric space with some positive $\sigma$-finite 
measure $\mu$, and $E$ is a Banach space, then $L^p(X,\mu;E)$ stands for the space of Bochner integrable with power $p$ 
 functions $f\colon X\to E$. The notation $L^p(X,\mu;\R^\infty)$ stand for the space of strongly measurable functions
$f\colon X\to \R^\infty$ such that $L^p(X,\mu;\R^m)$ for all $m\in \N$. In case $E=\R$ we will omit the reference to $E$ and write just $L^p(X,\mu)$  for the space of integrable with power $p$ 
real valued functions
(or just measurable essentially bounded in case $p=\infty$). By $\supp \mu$ we denote the support of the measure $\mu$.

We write $(\cdot,\cdot)_H$ for the scalar product in a Hilbert space $H$.
For a linear operator $T$ between Hilbert spaces we denote by $\norm{T}_{HS}$ its Hilbert-Schmidt norm, and by
$\norm{T}_{1}$ its trace class norm. If $\norm{T}_1<\infty$, then $T$ is called trace-class (or nuclear) operator.


\section{Multidimensional scaling}\label{sec:MDS}

\subsection{MDS of a finite sample}\label{sec:MDS_finite}

Given a metric space $(X, \dist)$ and a finite subset $\{ x_1, \ldots, x_n\} \subset  X$, we introduce the matrices
\begin{align}
\bar{K}_n &\eqdef \left(-\frac{1}{2 n} \dist^2(x_i, x_j)\right)_{i,j = 1,\ldots, n}, \label{eq:MDS_matrix_K}\\
\bar{P}_n &\eqdef \mathrm{Id}_n - \frac{1}{n} \one_n \one_n^\T, \label{eq:MDS_matrix_P}\\
\bar{T}_n &\eqdef \bar{P}_n \bar{K}_n \bar{P}_n \label{eq:MDS_matrix_T},
\end{align}
where $\mathrm{Id}_n$ stands for the identity $n\times n$ matrix, $\one_n$ stands for the column vector
$\one_n\eqdef (1,\ldots, 1)^\T$, $\T$ stands for the matrix transpose. 
Then the classical MDS 
consists in minimizing the strain function  defined by 
\begin{equation}\label{eq:strain_n}
\strain_n(y_1, \ldots, y_n) \eqdef \sum_{i, j = 1}^n \left((\bar{T}_n)_{i,j} - y_i\cdot y_j \right)^2
\end{equation}
among all $\{y_1, \ldots, y_n\} \subset \mathcal{H}$, where $\mathcal{H}$ is some chosen finite-dimensional Euclidean space \cite{adams-blumstein-kassab2020MDS,kassab2019MDS}. Note that one can easily rewrite
\[
\strain_n(y_1, \dots, y_n) = \norm{\bar{T}_n - G_y}_{HS}^2,
\]
where the positive symmetric $n\times n$ matrix $G_y$ is the Gram matrix of $y_1, \dots, y_n$.
Then minimization with respect to  $y_1, \dots, y_n$ is equivalent to minimization over positive-semidefinite matrices $G_y$.
In particular, one possible MDS maps from the finite set $\{x_1, \dots, x_n\}$ to $\R^n$ has the form 
\begin{equation}\label{eq:MDS_empirical}
\bar{M}_n(x_i) = \left(\sqrt{\lambda_k^+(\bar{T}_n)} [v_k^+(\bar{T}_n)]_i\right)_{k =1,\ldots,n},
\end{equation}
where $\lambda_k^+(\bar{T}_n)$ are positive eigenvalues of $\bar{T}_n$ (counting multiplicity) and $v_k^+(\bar{T}_n) \in \R^n$ are corresponding orthonormal eigenvectors.

\subsection{MDS of a metric measure space}\label{sec:MDS_metric_space}

Let $\mu$ be a Borel probability measure on $(X, \dist)$ with finite $4$th moment, i.e.
\begin{equation}\label{eq_d4mu1}
\int_{X} \dist^4(x_0, y) \,d \mu(y) < \infty
\end{equation}
for some $x_0\in X$.
Define now
\begin{align*}
\kap(x,y) &\eqdef - \frac{1}{2} 
\dist^2(x, y), \\ 
(\Kap f) (x) &\eqdef \int_X \kap(x, y) f(y) \,d \mu(y).
\end{align*}
Clearly, under condition~\eqref{eq_d4mu1} one has that $\kap \in L^2(X \times X, \mu \otimes \mu)$ and
hence the operator $\Kap \colon L^2(X,\mu)\to L^2(X,\mu)$ is a linear self-adjoint Hilbert--Schmidt (and hence compact)  
operator with
\[
\norm{\Kap}_{HS}^2= \frac 1 4 \int_{X \times X} \dist^4(x, y) \,d \mu \otimes \mu (x,y) < \infty
\]
because
\[
 \dist^2(x, y) \leq 8 \left(  \dist^2(x, x_0) +  \dist^2(x_0, y)\right). 
\]

Consider the projector operator $P\eqdef \mathrm{Id} - \one \otimes \one$
to the orthogonal complement of constant functions in $L^2(X,\mu)$,  
where $\mathrm{Id}$ stands for the identity operator in $L^2(X,\mu)$, and denote 
\[
T \eqdef P K P.
\]
Clearly $T$ is also an integral operator 
\begin{equation}\label{eq_defT1}
\begin{aligned}
(T f) (x) &= \int_X \kap_T(x, y) f(y) \,d \mu(y), \quad\mbox{where}\\
\kap_T(x,y) &\eqdef \kap(x, y) - \int_X \kap(x, y') \,d \mu(y') - \int_X \kap(x', y) \,d \mu(x') \\
& \qquad \qquad
+ \int_X \int_X \kap(x', y') \,d \mu(x') \,d \mu(y').
\end{aligned}
\end{equation}
Moreover, $T$ is also a self-adjoint Hilbert--Schmidt operator. It is immediate to see that zero is an eigenvalue
of $T$ whose eigenspace contains constant functions.


We will call in analogy with~\eqref{eq:MDS_empirical} 
the \textit{infinite MDS embedding} map from $X$ to $\R^\infty$ the map defined by
\begin{equation}\label{eq:MDS}
M(x) \eqdef \left(\sqrt{\lambda_j^+(T)} u_j^+(T)(x)\right)_{j \ge 1},
\end{equation}
where $\lambda_1^+(T) \ge \lambda_2^+(T) \ge \dots > 0$ are positive eigenvalues of $T$ (counting multiplicity), and $u_j^+(T) \in L^2(X,\mu)$ are corresponding orthonormal eigenfunctions. If the set of positive eigenvalues of $T$ contains $J< \infty$ elements, we set $(M(x))_j \eqdef 0$ for $j >J$.
In Proposition~\ref{prop_MDScorrect1} we prove that at least one eigenvalue of $T$ is positive, and thus infinite MDS embedding does not reduce to a singleton.
Note that $M$ is only defined once all the eigenfuctions $u_j^+(T)$ are chosen; a different choice of eigenfunctions gives a different map $M$.

It is conjectured in~\cite{kassab2019MDS} that this map is a limit of MDS embeddings of finite samples.
In what follows we rigorously formulate and prove this conjecture.

\subsection{Basic properties}

We start with the following easy calculation showing that eigenfunctions of $T$ are locally Lipschitz continuous.

\begin{lemma}\label{lm_MDSLip}
	Let $u \in L^2(X,\mu)$, $\norm{u}_{L^2(X,\mu)} = 1$ be an eigenfunction of $T$ with eigenvalue $\lambda \neq 0$, i.e., $T u = \lambda u$. Then for 
	any $x_0, x, x' \in X$, one has
\[
\abs{u(x) - u(x')} \le \frac{\dist(x, x')}{|\lambda|} \left[ \dist(x_0, x) + \dist(x_0, x') + 2 \left(\int_X \dist^2(x_0, y) \,d \mu(y)\right)^{1/2}\right] .
\]
In particular, if $\diam X < \infty$, then
\[
\abs{u(x) - u(x')} \le \frac{4 \diam X}{|\lambda|} \dist(x, x').
\]
\end{lemma}

\begin{proof}
One has
\begin{align*}
\abs{\lambda} \abs{u(x) - u(x')} &\le \int_X \abs{(\kap_T(x, y) - \kap_T(x', y)) u(y)} \,d \mu(y) \\
&\le \int_X \abs{\kap(x, y) - \kap(x', y)} \cdot \abs{u(y)} \,d \mu(y) \\
&\qquad 
+ \abs{\int_X \kap(x, y) \,d \mu(y) - \int_X \kap(x', y) \,d \mu(y)} \int_X \abs{u(y)} \,d \mu(y) \\
&\leq \int_X \abs{\kap(x, y) - \kap(x', y)} \cdot \abs{u(y)} \,d \mu(y)
+ \int_X \abs{\kap(x, y)  -  \kap(x', y)} \,d \mu(y)  \\
&\le 2 \left(\int_X \left(\kap(x, y) - \kap(x', y)\right)^2 \,d \mu(y)\right)^{1/2} \\
&= \left(\int_X \left(\dist^2(x, y) - \dist^2(x', y)\right)^2 \,d \mu(y)\right)^{1/2}.
\end{align*}
Furthermore,
\begin{align*}
\int_X &\left(\dist^2(x, y) - \dist^2(x', y)\right)^2 \,d \mu(y)
\le \int_X \dist^2(x, x') \left(\dist(x, y) + \dist(x', y)\right)^2 \,d \mu(y) \\
&\le \dist^2(x, x') \int_X \left(2 \dist(x_0, y) + \dist(x_0, x) + \dist(x_0, x')\right)^2 \,d \mu(y) \\
&= \dist^2(x, x') \norm{2 \dist(x_0, \cdot) + \dist(x_0, x) + \dist(x_0, x')}_{L^2(X,\mu)}^2,
\end{align*}
thus
\begin{align*}
\abs{\lambda} \abs{u(x) - u(x')} &\le \dist(x, x') \norm{2 \dist(x_0, \cdot) + \dist(x_0, x) + \dist(x_0, x')}_{L^2(X,\mu)} \\
&\le \dist(x, x')\left(\dist(x_0, x) + \dist(x_0, x') + 2 \norm{\dist(x_0, \cdot)}_{L^2(X,\mu)}\right),
\end{align*}
since $\mu(X)=1$ is a probability. The thesis follows.
\end{proof}


We now show that at least one positive eigenvalue always exists, thus the infinite MDS map $M$ is never trivial, and, moreover, 
$M$ is continuous when viewed as a map between $X$ and $\R^\infty$.

\
\begin{proposition}\label{prop_MDScorrect1}
Under condition~\eqref{eq_d4mu1} the operator
$T$ has at least one positive eigenvalue.
One has that $M\colon X\to \R^\infty$ is continuous
and, moreover, 
$M \in L^2(X, \mu; \ell^2)$ when $\sum_j \lambda_j^+(T) <+\infty$, which happens, for instance, when $T$ is a trace-class (i.e.\ nuclear) operator.  
\end{proposition}

\begin{remark}
	In particular examples one might know the asymptotics of the eigenvalues of $T$. For instance, if $X=\mathbb{S}^d$ is the $d$-dimensional sphere equipped with its intrinsic distance $\dist$ and the surface measure $\mu$, then~\eqref{eq:lambda_bound}
	gives the power asymptotics for positive eigenvalues of $T$, which implies according to the above Proposition~\ref{prop_MDScorrect1} that $M \in L^2(X, \mu; \ell^2)$. The same holds if instead
	$X$ is just a subset of a sphere not containing any couple of antipodal points ($\dist$ and $\mu$ being the same), because in this case the kernel of $T$ is smooth and hence the eigenvalues vanish quicker than any power (and this is of course not specific to a sphere, but rather holds for $X$ subset of a smooth Riemannian manifold with diameter $\diam X$ smaller than the injectivity radius of the latter, $\dist$ the Riemannian distance and $\mu$ a volume measure).
\end{remark}

\begin{proof}
If $T$ has only negative eigenvalues, i.e.\ $-T$ is positive definite, then by~\cite[Satz~1]{Wiedeman66}  
one has
\begin{align*}
0< \sum_j (-\lambda_j) & =  \tr (-T) \eqdef - \int_X \kap_T(x, x) \,d \mu(x) \\
&= -\int_X \kap(x, x) \,d \mu(x) + \int_X \int_X \kap(x, y) \,d \mu(x) \,d \mu(y) \\
&= - \frac{1}{2} \int_X \int_X \dist^2(x, y) \,d \mu(x) \,d \mu(y) < 0,
\end{align*}
a contradiction proving the existence of at least one positive eigenvalue of $T$. 

Continuity of $M\colon X\to \R^\infty$ follows from  Lemma~\ref{lm_MDSLip}.
Moreover, 
under additional assumption that 
\[
\sum_{j=1}^\infty \lambda_j^+(T) < \infty,
\]
for every $s\in \ell^2$ one has 
the relationship
\[
(s, M(x))_{\ell^2}= \sum_j s_j \sqrt{\lambda_j^+(T)} u_j^+(T)(x),
\] 
and the series 
\[
\sum_j \lambda_j^+(T)(u_j^+(T))^2(x) 
\]
is convergent for $\mu$-a.e.\ $x\in X$, because
\[
\int_X \sum_{j=1}^m \lambda_j^+(T)(u_j^+(T))^2(x)\, d\mu(x) =\sum_{j=1}^m \lambda_j^+(T) 
<+\infty.
\]
Thus one has
\begin{align*}
\lim_{m\to \infty} \abs{\sum_{j> m} s_j \sqrt{\lambda_j^+(T)} u_j^+(T)(x)} & \leq 
\lim_{m \to \infty} \left(\sum_{j> m} s_j^2\right)^{1/2} \left(\sum_{j > m} \lambda_j^+(T)(u_j^+(T))^2(x)\right)^{1/2} \\
& = 0,
\end{align*}
and therefore, 
\[
(s,M(x))_{\ell^2}= \lim_{m \to \infty} \sum_{j =1}^m s_j \sqrt{\lambda_j^+(T)} u_j^+(T)(x)
\]
is well-defined, for $\mu$-a.e.\ $x\in X$, hence the function $s\in \ell^2 \mapsto (s, M(x))_{\ell^2}\in \R$ is measurable as a $\mu$-a.e.\ limit
of a sequence of continuous functions. This means that $M\colon X\to \ell^2$ is weakly (hence also strongly) measurable,
and finally the relationship
\[
\int_X |M(x)|^2\, d\mu(x)=\sum_j \lambda_j^+(T) 
<\infty
\]
shows that in this case
$M \in L^2(X, \mu; \ell^2)$.
\end{proof}

Calculation of the infinite MDS embedding map for a space homogeneous under the action of some group by measure preserving isometries
may be simplified with the help of the following statement.

\begin{proposition}\label{prop_MDShomogeneous}
	Let $(X, \dist, \mu)$ be homogeneous, i.e.\ there is a group $G$ acting transitively on $X$ by isometries preserving $\mu$. Then $\one$ is an eigenfunction of $K$, hence $T$ and $K$ share all eigenfunctions and eigenvalues except for the eigenvalue corresponding to $\one$.
\end{proposition}

\begin{proof}
	Fix $x_0 \in X$ and take an arbitrary $x \in X$. There exists $g \in G$ such that $x_0 = g x$. Then
	\begin{align*}
	(K \one) (x) &= \int_X \kap(x, y) \,d \mu(y)
	= \int_X \kap(g x, g y) \,d \mu(y) \\
	&= \int_X \kap(x_0, y) \,d (g_\# \mu)(y)
	= \int_X \kap(x_0, y) \,d \mu(y)
	= (K \one) (x_0),
	\end{align*}
that is, $\one$ is an eigenfunction of $K$ corresponding to the eigenvalue $\lambda\eqdef (K \one) (x_0)$.
\end{proof}

\section{Injectivity almost everywhere}\label{sec:inject}

We let $\{\lambda_i\}_{i= 1}^\infty$ denote all the eigenvalues of the operator $T$ defined over $L^2(X,\mu)$ by~\eqref{eq_defT1} and by $u_i \in L^2(x,\mu)$ the eigenfunctions corresponding to $\lambda_i$, chosen so as to form an 
orthonormal basis $\{u_i\}_{i=1}^\infty$ in $L^2(X,\mu)$. 
One has 
\[
\sum_{i= 1}^\infty \lambda_i u_i \otimes u_i = \kap_T 
\]
the convergence being understood in the sense of $L^2(X\times X,\mu\otimes \mu)$, and even more, since both $\kap_T$ and all $u_i$ are continuous if $\lambda_i \neq 0$, then in fact
\[
\sum_{i = 1}^\infty \lambda_i u_i(x) u_i(\cdot) = k_T(x,\cdot) 
\]
in the sense of $L^2(X,\mu)$ for every $x\in X$.

\begin{lemma}\label{lm_MDSLebpt1}
Let $T$ be defined over $L^2(X,\mu)$ by~\eqref{eq_defT1}, and assume that $T$ is a trace-class (or nuclear) operator, i.e.
\begin{equation}\label{eq_Tnucl1}
\norm{T}_1 \eqdef \sum_i |\lambda_i| < +\infty.
\end{equation}
Define the function
\begin{equation}
f \eqdef \sum_{i = 1}^\infty \abs{\lambda_i} u_i^2 \in L^1(X,\mu) .
\end{equation}
	If $x \in X$ is a Lebesgue point of $f$, then
	\[
	\sum_{i = 1}^\infty \lambda_i u_i^2(x) = k_T(x, x) .
	\]
\end{lemma}

\begin{proof}
	Fix $\varepsilon > 0$. Take $n \in \N$ such that 
	\[
	\sum_{i = n}^\infty \abs{\lambda_i} u_i^2(x)< \varepsilon.
	\]
	Now notice that
	\begin{align*}
	\sum_{i = 1}^\infty \lambda_i \left(\fint_{B_r(x)} u_i \,d \mu\right)^2 
	&= \sum_{i = 1}^\infty \lambda_i \fint_{B_r(x) \times B_r(x)} u_i \otimes u_i \,d \mu \otimes \mu \\
	&= \fint_{B_r(x) \times B_r(x)} \left(\sum_{i = 1}^\infty \lambda_i u_i \otimes u_i\right) \,d \mu \otimes \mu \\
	&= \fint_{B_r(x) \times B_r(x)} \kap_T \,d \mu \otimes \mu \to \kap_T(x, x) \text{~~as~~} r \to 0.
	\end{align*}
	Furthermore, since all functions $u_i$ are continuous,
	\[
	\sum_{i = 1}^{n - 1} \lambda_i \left(\fint_{B_r(x)} u_i \,d \mu\right)^2 \to \sum_{i = 1}^{n - 1} \lambda_i u_i^2(x)  \quad \text{as } r \to 0,
	\]
	and since $x$ is a Lebesgue point of $f$,
	\begin{align*}
	\abs{\sum_{i = n}^\infty \lambda_i \left(\fint_{B_r(x)} u_i \,d \mu\right)^2}
	&\le \sum_{i = n}^\infty \abs{\lambda_i} \fint_{B_r(x)} u_i^2 \,d \mu \\
	&= \fint_{B_r(x)} f \,d \mu - \sum_{i = 1}^{n - 1} \abs{\lambda_i} \fint_{B_r(x)} u_i^2 \,d \mu \\
	&\to f(x) - \sum_{i = 1}^{n - 1} \abs{\lambda_i} u_i^2(x) = \sum_{i = n}^\infty \abs{\lambda_i} u_i^2(x) < \varepsilon \text{~~as~~} r \to 0.
	\end{align*}
	Combining the above results, we conclude.
\end{proof}

\begin{theorem}\label{theo:d2_decomp}
	Assume that~\eqref{eq_Tnucl1} holds and let $\mu$-a.e.\ point $x \in X$ be a Lebesgue point of $f$. Then
	\begin{equation}\label{eq_aeinj1}
		\sum_{i = 1}^\infty \lambda_i \left(u_i(x) - u_i(y)\right)^2 = 
\dist^2(x, y) \quad 	\text{for $\mu \otimes \mu$-a.e.\ $(x,y)$}.
	\end{equation}
\end{theorem}

\begin{remark}\label{rem:vitali}
	The assumption on $f$ from the above theorem is fulfilled, e.g., if $\mu$ is doubling, i.e.\ there is a constant $C > 0$ such that for any $x \in X$ and $r > 0$ one has
	\[
	\mu\left(B_{2 r}(x)\right) \le C \mu\left(B_{r}(x)\right)
	\]
	(see~\cite[theorem~5.2.3]{AmbrTilli00} or  alternatively~\cite[corollary~2.9.9, theorem~2.8.17]{federer1969geometric})).
	In fact, it holds also as well on infinitesimally doubling spaces, also called Vitali spaces, see~\cite[theorem 3.4.3]{heinonen2015sobolev}, a class that includes any smooth Riemannian manifold.
\end{remark}

\begin{remark}
	It is also worth remarking that under condition~\eqref{eq_Tnucl1} one has
	\begin{equation*}\label{eq_trT2}
		\sum_{i = 1}^\infty \lambda_i = \tr (T) \eqdef \int_X \kap_T(x,x)\, d\mu(x)	= \frac 1 2 \int_X \int_X \dist^2(x,y)\, d\mu(x) d\mu(y).
	\end{equation*}
\end{remark}

\begin{proof}
	From Lemma~\ref{lm_MDSLebpt1} we get
	\begin{align*}
	\sum_{i = 1}^\infty \lambda_i \left(u_i(x) - u_i(y)\right)^2 
	&= \sum_{i = 1}^\infty \lambda_i u_i^2(x) + \sum_{i = 1}^\infty \lambda_i u_i^2(y) - 2 \sum_{i = 1}^\infty \lambda_i u_i(x) u_i(y) \\
	&= \kap_T(x, x) + \kap_T(y, y) - 2 \kap_T(x, y)
	\end{align*}
	for $\mu \otimes \mu$-a.e.\ $(x,y)$. The identity
	\[
	\kap_T(x, x) + \kap_T(y, y) - 2 \kap_T(x, y) = \kap(x, x) + \kap(y, y) - 2 \kap(x, y) = \dist^2(x, y)
	\]
	then concludes the proof.
\end{proof}

\begin{remark}\label{rm_positiveMDS}
If under conditions of Theorem~\ref{theo:d2_decomp} one has additionally that $T$ is positive semidefinite (i.e. $\lambda_i\geq 0$ for all $i\in \N$), then in view of Mercer's theorem one has that~\eqref{eq_aeinj1} holds
for all  $(x,y)\in \supp\mu\times \supp\mu$ rather than just for $\mu\otimes\mu$-a.e. $(x,y)$. In other words, $M$ is an isometric embedding of
$\supp\mu \subset X$ into $\ell^2$.
\end{remark}

Now we define a new map
\[
M^{-}(x) \eqdef \left(\sqrt{\lambda_j^- (T)} u_j^-(T)(x)\right)_{j \ge 1},
\]
where $\lambda_j^-(T)$ and $u_j^-(T)$ are the absolute values of negative eigenvalues of $T$ and the respective orthonormal eigenfunctions (with the same convention as with $M$ if these are a finite number). In the same way as $M$, one has $M^{-} \in L^2(X, \mu; \ell^2)$. Theorem~\ref{theo:d2_decomp} ensures that for $\mu \otimes \mu$-a.e.\ $(x, y)$
\begin{equation}\label{eq_MMinj1}
\norm{M(x) - M(y)}_2^2 - \norm{M^{-}(x) - M^{-}(y)}_2^2 = \dist^2(x, y).
\end{equation}

As an important corollary we have the following statement.

\begin{corollary}\label{co_MMinj1}
	Under conditions of Theorem~\ref{theo:d2_decomp} one has
\begin{equation}\label{eq_MDSinj2}
\norm{M(x) - M(y)}_2^2\geq \dist^2(x, y)
\end{equation}
for $\mu \otimes \mu$-a.e.\ $(x, y)$. 

Moreover, the map $x \mapsto N(x) \eqdef (M(x), M^{-}(x)) \in \mathcal{K}$ considered as a map between $X$ and the Krein space 
$\mathcal{K} \eqdef \ell^2 \times \ell^2$
with the indefinite inner product
\[
\left( (f_1, g_1), (f_2, g_2) \right)_{\mathcal{K}} \eqdef \left( f_1, f_2 \right) - \left( g_1, g_2 \right)
\]
and the corresponding pseudo-norm
\[
\norm{(f, g)}_{\mathcal{K}}^2 \eqdef \norm{f}_2^2 - \norm{g}_2^2 ,
\]
is an a.e.\ isometry
in the sense that
\begin{equation}\label{eq_MMinj2}
\norm{N(x) - N(y)}_{\mathcal{K}}^2 = \dist^2(x, y)
\end{equation}
for $\mu \otimes \mu$-a.e.\ $(x, y)\in X\times X$.
\end{corollary}

\begin{remark}\label{rem:full-injectivity}
Under conditions of Theorem~\ref{theo:d2_decomp} and hence also of Corollary~\ref{co_MMinj1} one has
$M\in L^2(X,\mu;\ell^2)$. If one knows additionally that 
\begin{equation}\label{eq_MDSunif1}
\lim_m \sup_{x\in X} \sum_{j\geq m}\lambda_j^+(T) (u_j^+(T))^2(x) =0,
\end{equation}
then more can be said, namely,~\eqref{eq_MDSinj2} holds for all 
$(x, y)\in X\times X$ rather than just for $\mu\otimes\mu$-a.e., i.e.\ in particular $M$ is injective. 
When $X$ is compact, this also implies that $M\colon X\to \R^\infty$ is a homeomorphism onto its image
since it is a continuous map.
In fact, by~\eqref{eq_MDSunif1} for every $\varepsilon>0$ there is an $m\in \N$ such that
\[
\norm{M(x) - M(y)}_2^2\leq |M^m(x) - M^m(y)|^2 +\varepsilon
\]
for all 
$(x, y)\in X\times X$,
where $M^m\colon X\to \R^m$ stands for the restriction of $M$ to its first $m$ components, because
\begin{align*}
0 \leq \norm{M(x) - M(y)}_2^2 & - |M^m(x) - M^m(y)|^2  \\ 
& = \sum_{j\geq m}\lambda_j^+(T) (u_j^+(T))^2(x) + \sum_{j\geq m}\lambda_j^+(T) (u_j^+(T))^2(y) \\
&\qquad\qquad- 2\sum_{j\geq m}\lambda_j^+(T) (u_j^+(T))(x) (u_j^+(T))(y) \\
& \leq 2 \sum_{j\geq m}\lambda_j^+(T) (u_j^+(T))^2(x) + 2 \sum_{j\geq m}\lambda_j^+(T) (u_j^+(T))^2(y).
\end{align*}
Thus by~\eqref{eq_MDSinj2} one would have
\[ 
\dist^2(x,y)\leq |M^m(x) - M^m(y)|^2 +\varepsilon
\]
for $\mu \otimes \mu$-a.e.\ $(x, y)\in X\times X$, and minding that $M^m$ is continuous by Lemma~\ref{lm_MDSLip}, in fact, also for
all 
$(x, y)\in X\times X$. Letting then $\varepsilon\to 0$ we get the desired claim.

The condition~\eqref{eq_MDSunif1} is satisfied, for instance, when one has 
\begin{equation*}\label{eq_MDSinj3}
\norm{u_j^+(T)}_{L^\infty(X,\mu)} \leq C\lambda_j^{-\alpha} \quad \text{and} \quad \sum_{j=1}^\infty (\lambda_j^+(T))^{1-2\alpha} <+\infty
\end{equation*} 
for some $C>0$ and $\alpha\in [0,1/2)$. 
Another interesting particular case is when the range of $M$ is some $\R^m$, i.e.\ $M = (M^m, 0, \dots)$, or the operator $T$ has only $m \in \N$ positive eigenvalues.
Combining~\eqref{eq_MDSinj2} with Lemma~\ref{lm_MDSLip} one has that in the latter case
\[
\dist(x,y)\leq |M(x) - M(y)| \leq C \dist(x,y)
\]
for all 
$(x, y)\in X\times X$, i.e.\ in other words $M$ provides a bi-Lipschitz embedding of $X$ into $\R^m$.
\end{remark}

%

Let us finally mention that Corollary~\ref{co_MMinj1} actually holds for any continuous symmetric kernel $\kap$, not necessarily the negative squared distance, with $\dist^2(\cdot, \cdot)$ in all assertions replaced with $- 2 \kap(\cdot, \cdot)$.

\section{Stability of MDS}\label{sec:stability}

\subsection{Stability under Gromov--Kantorovich convergence}\label{sec:stability_GW}

Let $(X_n, d_n, \mu_n)$ be a sequence of separable complete metric measure spaces. Recall that the Gromov--Kantorovich distance\footnote{It is nowadays customary to call these distances and the respective convergence Gromov--Wasserstein, but we would like to avoid this historically incorrect naming.} of order $p$ is defined as~\cite[definition~5.7]{memoli2011gromov}
\begin{equation}\label{eq_GWp1}
GW_p^p(X, X_n) \eqdef \inf_{\gamma \in \Pi(\mu, \mu_n)} \int_{(X \times X_n)^2} \abs{\dist(x, y) - \dist_n(x', y')}^p \,d \gamma(x, x') \,d \gamma(y, y'),
\end{equation}
where  $\Pi(\mu, \mu_n)$ denote the set of transport plans between $\mu$ and $\mu_n$, i.e., joint probability densities on $X \times X_n$ with marginals respectively given by $\mu$, $\mu_n$. Assume $\lim_n GW_4(X, X_n) = 0$. We are going to show that in this case 
\[
\lim_n GW_2\bigl(M^m(X), M_n^m (X_n)\bigr) =0
\]
for every $m\in \N$,
where $M^m$ (resp.\ $M_n^m$) is any MDS map of $(X, \dist, \mu)$ (resp.\ $(X_n, \dist_n, \mu_n)$) into $\R^m$.

Denote by $\gamma_n \in \Pi(\mu, \mu_n)$ any optimal transport plan in~\eqref{eq_GWp1}. 
According to \cite[lemma~5.3.4]{ambrosio2008gradient} there exists a probability measure $\bm{\gamma}$ on the product space
\[
X^\infty \eqdef X \times \prod_n X_n,
\]
 such that its projection on any pair $X \times X_n$ is equal to $\gamma_n$.
Note that $X^\infty$ is still
Polish as a countable product of Polish spaces, e.g.\ with the metric 
\[\dist_\infty(\bm{x}, \bm{y}) \eqdef \dist(x, y) + \sum_{n = 1}^\infty 2^{-n} \frac{\dist_n(x_n, y_n)}{1 + \dist_n(x_n, y_n)}.\]
Embeddings of $L^2(X,\mu)$ and $L^2(X_n, \mu_n)$ to $L^2(X^\infty, \bm{\gamma})$ also allow us to naturally extend the operators $K$, $K_n$, $T$, $T_n$ on this space, with $K$, $T$ vanishing on $\left(L^2(X, \mu)\right)^\bot$, and $K_n$, $T_n$ vanishing on $\left(L^2(X_n, \mu_n)\right)^\bot$, respectively. Note that $K$ is induced by the kernel $\kap(\bm{x}, \bm{y}) = - \frac{1}{2} \dist^2(x, y)$ and $K_n$ is induced by the kernel $\kap_n(\bm{x}, \bm{y}) = - \frac{1}{2} \dist_n^2(x_n, y_n)$. Moreover, it is straightforward to check that
\[
T = P_{\gamma_n} K P_{\gamma_n}, \quad T_n = P_{\gamma_n} K_n P_{\gamma_n}
\]
where $P_{\gamma_n} \eqdef \mathrm{Id} - \one \otimes \one$ is an orthogonal projector in $L^2(X^\infty, \bm{\gamma})$.
Clearly, this embedding does not change any eigenfunctions of $T$ and $T_n$ corresponding to non-zero eigenvalues, thus we can equivalently define the MDS maps $M^m$ and $M_n^m$ using these embedded operators.
The next lemma ensures that $\norm{T_n - T}_{HS} \to 0$ as $GW_4(X, X_n) \to 0$.

\begin{lemma}\label{lem:GW_to_HS}
Setting $C_X = \norm{\dist(\cdot, \cdot)}_{L^4(X \times X, \mu \otimes \mu)}$, one has
	\[
	\norm{T_n - T}_{HS} \le \norm{K_n - K}_{HS} 
	\le  C_X GW_4(X, X_n) + \frac{1}{2} GW_4^2(X, X_n).
	\]
\end{lemma}

\begin{proof}
	Clearly, it is enough to consider the restrictions of the operators to $L^2(X \times X_n, \gamma_n)$, with $\gamma \in \Pi(\mu, \mu_n)$:
	\begin{align*}
	\norm{T_n - T}_{HS} & \le \norm{K_n - K}_{HS} 
	= \norm{\kap_n - \kap}_{L^2((X \times X_n)^2, \gamma_n \otimes \gamma_n)} \\
	&=  \frac{1}{2} \left( \int_{(X \times X_n)^2} (\dist_n^2(x_n, y_n) - \dist^2(x, y))^2  d \gamma_n (x, x_n) d \gamma_n (y, y_n)\right)^{1/2} \\
	&\le \left( \int_{(X \times X_n)^2} \dist(x, y) \left(\dist_n(x_n, y_n) - \dist(x, y)\right) d \gamma_n (x, x_n) d \gamma_n (y, y_n) \right)^{1/2} \\
	& \qquad + \frac{1}{2} \left( \int_{(X \times X_n)^2}  \left(\dist_n(x_n, y_n) - \dist(x, y)\right)^2 d \gamma_n (x, x_n) d \gamma_n (y, y_n) \right)^{1/2}  \\
	&\le \left( \int_{X \times X} \dist(x, y)^4 d \mu\otimes \mu (x,y)  \right)^{1/4} GW_4(X, X_n) + \frac{1}{2} GW_4^2(X, X_n).\qedhere
	\end{align*}
\end{proof}

Now let us recall some stability results for compact self-adjoint operators on a Hilbert space (see \cite{rosasco2010learning}).

\begin{proposition}[Kato 1987]\label{prop:spectrum_stab}
	For any self-adjoint compact operators $A$, $B$ on a Hilbert space $\mathcal{H}$ there are enumerations $\{\alpha_i\}_i$, $\{\beta_i\}_i$ of their eigenvalues (counting multiplicity) such that
	\[
	\sum_{i = 1}^\infty \abs{\alpha_i - \beta_i}^2 \le \norm{A - B}_{HS}
	\]
	and, with possibly another order of eigenvalues,
	\[
	\sup_i \abs{\alpha_i - \beta_i} \le \norm{A - B}_{HS}.
	\]
\end{proposition}

\begin{proposition}\label{prop:proj_stab}
	Let $A$ and $B$ be compact self-adjoint operators on a Hilbert space $\mathcal{H}$. Take $k \in \N$ such that $\lambda_k(A) \neq 0$ and define $r_k \eqdef \frac{1}{2} d\bigl(\lambda_k(A), \sigma(A) \setminus \{\lambda_k(A)\}\bigr)$. If $\norm{A - B}_{HS} \le \frac{r_k}{2}$, then
	\[
	\norm{P_A - P_B}_{HS} \le \frac{2}{r_k} \norm{A - B}_{HS},
	\]
	where $P_A$ and $P_B$ are the orthogonal projectors onto $\mathrm{span}\left\{u_i(A) : \lambda_i(A) = \lambda_k(A)\right\}$ and $\mathrm{span}\left\{u_i(B) : \abs{\lambda_i(B) - \lambda_k(A)} \le r_k\right\}$, respectively. Moreover, the dimensions of their ranges coincide.
\end{proposition}

\begin{proof}
	The result follows from theorem~20 in \cite{rosasco2010learning} (recall that for self-adjoint operators spectral projectors considered in that theorem are orthogonal). Indeed, take as a curve $\Gamma \subset \mathbb{C}$ enclosing $\lambda_k(A)$ the circle of radius $r_k$ centered at $\lambda_k(A)$, so 
	\[
	\delta \eqdef \min_{\lambda \in \sigma(A), z \in \Gamma} \abs{z - \lambda} = r_k
	\] 
	and $\ell(\Gamma) = 2 \pi \delta$. Then
	\[
	\norm{P_A - P_B} \le \frac{\ell(\Gamma)}{2 \pi \delta} \frac{\norm{A - B}}{\delta - \norm{A - B}} 
	\le \frac{2 \norm{A - B}}{\delta} 
	\]
	and the dimensions of their ranges coincide.
\end{proof}

Now we are ready to prove the main result of this section.

\begin{theorem}\label{theo:GW_stability}
	Let $\lim_n GW_4(X, X_n) = 0$. If $m \in \N$ is such that $\lambda^+_m(T) > \lambda^+_{m+1}(T)$, then 
	\[
	\min_{Q \in O(m)} \norm{M^m - Q M^m_n}_{L^2(X^\infty,\bm{\gamma})} \to 0
	\]
	as $n \to \infty$, where $O(m)$ is the orthogonal group on $\R^m$. 
	Moreover, if $\deg \lambda^+_j(T) = 1$ for all $1 \le j \le m$, then for every $n$ there exists some diagonal matrix $D_n \in \R^{m\times m}$ with diagonal elements in $\{-1, 1\}$, such that
	\[
	\lim_{n\to \infty} \norm{M^m - D_n M^m_n}_{L^2(X^\infty,\bm{\gamma})} = 0.
	\]
\end{theorem}

\begin{proof}
	Fix some $\lambda_j^+(T) > 0$ and consider the projectors $P_T$ and $P_{T_n}$ coming from Proposition~\ref{prop:proj_stab}.
	Let $P_T = \sum_{i \in I} u_i(T) \otimes u_i(T)$ for a suitable index set $I$ of size $d = \dim P_T$. For simplicity we can assume without loss of generality that $I = \{1, \dots, d\}$.
	Lemma~\ref{lem:GW_to_HS} implies that $\norm{T_n - T}_{HS} \to 0$.
	Then Propositions~\ref{prop:spectrum_stab} and~\ref{prop:proj_stab} yield that $P_{T_n} \to P_T$, and for $n$ large enough $P_{T_n} = \sum_{i \in I} u_i(T_n) \otimes u_i(T_n)$. Since
	\[
	\norm{(P_T - P_{T_n}) u_i(T_n)} = \norm{P_T u_i(T_n) - u_i(T_n)} \le \norm{P_T - P_{T_n}} \to 0,
	\]
	$\{u_i(T_n)\}_{i \in I}$ converges, up to a subsequence, to an orthonormal basis of $\operatorname{range}(T) = \mathrm{span}\{u_i(T) : i \in I\}$. Since all orthonormal bases are equivalent up to an orthogonal transformation, we obtain that in fact there are $Q^n \in O(d)$ such that for all $i = 1, \dots, d$,
	\[
	\sum_{j = 1}^d Q^n_{ij} u_j(T_n) \to u_i(T) \text{~~in~~} L^2(X^\infty,\bm{\gamma}).
	\]
	
	Combining the above result with the convergence of the spectra (Proposition~\ref{prop:spectrum_stab}), we conclude that
	\[
	\min_{Q \in O(m)} \norm{M^m - Q M^m_n}^2_{L^2(X^\infty,\bm{\gamma})} \to 0.
\qedhere	\]
\end{proof}

\begin{remark}
	We require $\lambda^+_m(T) > \lambda^+_{m+1}(T)$ to ensure that the MDS map $M^m$ is uniquely defined (up to an orthogonal transformation).
\end{remark}

\begin{corollary}\label{cor:GW_stability}
	Under the assumptions of Theorem~\ref{theo:GW_stability} for every $m\in \N$ one has
	\[
	GW_2\bigl(M^m(X), M_n^m (X_n)\bigr) \to 0
	\]
	as $n\to \infty$, where $M^m(X)$ stands for the space $(M^m(X), \dist_m, M^m_\#\mu)$ and  $M_n^m (X_n)$
	stands for $(M_n^m(X_n), \dist_m, (M^m_n)_\#\mu_n)$, $\dist_m$ standing for the usual Euclidean distance in $\R^m$.
\end{corollary}

\begin{proof}
	We bound the Gromov--Kantorovich distance between the MDS images with the distance between the maps themselves in $L^2(X^\infty,\bm{\gamma})$: for any $Q \in O(m)$,
	\begin{align*}
	&GW_2^2 \bigl(M^m(X), M_n^m (X_n)\bigr) \\
	&\quad \le \int\limits_{(X \times X_n)^2} \left(|M^m(x) - M^m(y)| - |M^m_n(x') - M^m_n(y')|\right)^2 \,d \gamma_n(x, x') \,d \gamma_n(y, y') \\
	&\quad = \int\limits_{(X \times X_n)^2} \left(|M^m(x) - M^m(y)| - |Q M^m_n(x') - Q M^m_n(y')|\right)^2 \,d \gamma_n(x, x') \,d \gamma_n(y, y') \\
	&\quad \le \int\limits_{(X \times X_n)^2} \left(|M^m(x) - Q M^m_n(x')| + |M^m(y) - Q M^m_n(y')|\right)^2 \,d \gamma_n(x, x') \,d \gamma_n(y, y') \\
	&\quad \le 4 \int\limits_{X \times X_n} |M^m(x) - Q M^m_n(x')|^2 \,d \gamma_n(x, x') \\
	&\quad = 4 \norm{M^m - Q M^m_n}^2_{L^2(X \times X_n, \gamma_n)} = 4 \norm{M^m - Q M^m_n}^2_{L^2(X^\infty, \bm{\gamma})}.
	\end{align*}
	Using Theorem~\ref{theo:GW_stability}, we conclude the proof.
\end{proof}

\subsection{Stability under Kantorovich convergence}\label{sec:stability_W}

In this section we consider the case where $\mu$ and $\mu_n$ are defined on the same space $X$, and $\mu_n$ converge to $\mu$ in the Kantorovich distance $W_4$.

\subsubsection{Embedding of a finite sample to $L^2(X,\mu)$.}

Assume again we are given points $x_1, \dots, x_n \in X$. Define the empirical measure $\mu_n = \frac{1}{n} \sum_{i=1}^n \delta_{x_i}$.
Suppose that an optimal transport plan from $\mu$ to $\mu_n$ with a cost function $\dist^4(\cdot, \cdot)$ gives us a  partition of $X$ into disjoint sets $V_1, \dots, V_n$, such that any $x \in V_i$ is transported to $t_n(x) := x_i$, i.e.\ there exists a solution to the Monge optimal transportation problem (which is the case e.g., if $\mu$ has no atoms). Define the operator
\[
A_n \colon L^2(X,\mu) \to \R^n,\qquad A_n f \eqdef \sqrt{n} \left(\int_{V_i} f \,d \mu \right)_{i =1,\ldots,n}.
\]
 It is easy to see that the adjoint operator $A_n^* \colon \R^n \to L^2(X, \mu)$ is given by
\[
A_n^* x = \sqrt{n} \sum_{i = 1}^n x_i \one_{V_i} .
\]
In particular, $A_n A_n^* = \mathrm{Id}_n$ since $\mu(V_i) = \frac{1}{n}$, hence $A_n^*$ is an isometric embedding of $\R^n$ into $L^2(X,\mu)$.
Furthermore, $A_n \one = \frac{1}{\sqrt{n}} \one_n$, $A_n^* \one_n = \sqrt{n} \one$, and
\[
A_n (\one \otimes \one) = \frac{1}{\sqrt{n}} \one_n \otimes \one, \quad
\one_n \one_n^\T A_n = \one_n \otimes (A_n^* \one_n) = \sqrt{n} \one_n \otimes \one.
\]
Thus,
\[
A_n P = A_n (\mathrm{Id} - \one \otimes \one) 
= A_n - \frac{1}{\sqrt{n}} \one_n \otimes \one
= A_n - \frac{1}{n} \one_n \one_n^\T A_n
= \bar{P}_n A_n .
\]
Now we consider the operators
\[
K_n \eqdef A_n^* \bar{K}_n A_n \text{~~and~~} T_n \eqdef A_n^* \bar{T}_n A_n = P K_n P ,
\]
where $\bar{K}_n$ and $\bar{T}_n$ are defined by \eqref{eq:MDS_matrix_K} and \eqref{eq:MDS_matrix_T} respectively.
The eigenvalues of $T_n$ are $\lambda_1(\bar{T}_n), \dots, \lambda_n(\bar{T}_n)$, all the remaining eigenvalues being zero and the first $n$ eigenfunctions are given by 
$u_j(T_n) = A_n^* v_j(\bar{T}_n)$, where $v_j(\bar{T}_n)$ is an eigenvector of $\bar{T}_n$. Indeed,
\[
T_n u_j(T_n)
= n A_n^* \bar{T}_n A_n A_n^* v_j(\bar{T}_n) 
= A_n^* \bar{T}_n v_j(\bar{T}_n)
= \lambda_j(\bar{T}_n) A_n^* v_j(\bar{T}_n)
= \lambda_j(\bar{T}_n) u_j(T_n) .
\]
Note that $\norm{u_j(T_n)}_{L^2(X,\mu)} = \abs{v_j(\bar{T}_n)}$,
this allows us to define an empirical MDS map $\mu$-a.e.\ by the formula
\begin{equation}\label{eq:MDS_n}
M_n(x) \eqdef \left(\sqrt{\lambda_j^+(T_n)} u_j^+(T_n)(x)\right)_{j \in \N}.
\end{equation}
Clearly, the metric measure spaces obtained by $\bar{M}_n$ and $M_n$ from $\{x_1, \dots, x_n\}$ and $X$, respectively, are the same.

\subsubsection{Convergence}

One can see that $K_n$ is also an integral operator, that is,
\[
K_n f (x) = \int_X k_n(x, y) f(y) \,d \mu(x),
\]
where $k_n(x, y) \eqdef k(t_n(x), t_n(y))$.
The following lemma controls $\norm{T_n - T}_{HS}$ with $W_4(\mu_n, \mu)$ similarly to Lemma~\ref{lem:GW_to_HS}.

\begin{lemma}\label{lem:W_to_HS}
	One has, setting $C_X = \norm{\dist(\cdot, \cdot)}_{L^4(X \times X, \mu \otimes \mu)}$, 
	\[
	\norm{T_n - T}_{HS} 	\le 2 C_X W_4(\mu_n, \mu) + 2 W_4^2(\mu_n, \mu).
	\]
\end{lemma}

\begin{proof}
	First of all,
	\begin{align*}
		\norm{T_n - T}_{HS} & \le \norm{K_n - K}_{HS} 
		= \norm{k_n - k}_{L^2(X \times X, \mu \otimes \mu)} \\
		&= \frac{1}{2} \left( \int_{X \times X} \left( \dist^2(t_n(x), t_n(y)) - \dist^2(x, y)\right)^2 d \mu \otimes \mu (x,y)\right)^{1/2}.
	\end{align*}
	Note that
	\begin{multline*}
		\abs{\dist^2(t_n(x), t_n(y)) - \dist^2(x, y)} \\
		\le 2 \dist(x, y) \abs{\dist(t_n(x), t_n(y)) - \dist(x, y)} + \abs{\dist(t_n(x), t_n(y)) - \dist(x, y)}^2
	\end{multline*}
	and
	\begin{align*}
		& \left( \int_{X \times X} \left( \dist(t_n(x), t_n(y)) - \dist(x, y)\right)^4 d \mu \otimes \mu (x,y) \right)^{1/4}\\
		&\qquad \le \left( \int_{X \times X} \left(\dist(x, t_n(x)) + \dist(y, t_n(y))\right)^4 d \mu \otimes \mu (x,y)\right)^{1/4} \\
		&\qquad \le 2\left( \int_{X} \dist(x, t_n(x))^4 d \mu(x) \right)^{1/4} = 2 W_4(\mu_n, \mu).
	\end{align*}
	Therefore, we have by the triangle and H{\"o}lder inequalities that
	\begin{align*}
		\MoveEqLeft[8]\left( \int_{X \times X} \left( \dist^2(t_n(x), t_n(y)) - \dist^2(x, y)\right)^2 d \mu \otimes \mu (x,y) \right)^{1/2} \\
		&\le 2 \left( \int_{X \times X} \left(\dist(x, y) [\dist(t_n(x), t_n(y)) - \dist(x, y)]\right)^2 d \mu \otimes \mu (x,y) \right)^{1/2} \\
		&~~~+ \left( \int_{X \times X} \left(\dist(t_n(x), t_n(y)) - \dist(x, y)\right)^4 d \mu \otimes \mu (x,y) \right)^{1/2} \\
		&\le 4 \left( \int_{X \times X} \dist(x, y)^4 d \mu\otimes \mu(x,y)\right)^{1/4} W_4(\mu_n, \mu) + 4 W_4^2(\mu_n, \mu).
	\end{align*}
\end{proof}

In the same way as in Section~\ref{sec:stability_GW} we now obtain the following statement.

\begin{theorem}\label{theo:W_stability}
	Let $\lim_n W_4(\mu_n, \mu) =0$. If $\lambda^+_m(T) > \lambda^+_{m+1}(T)$, then (up to an orthogonal transformation of $\R^m$)
	\[
	M^m_n \to M^m \text{ in } L^2(X, \mu; \R^m) 
	\]
	as $n\to\infty$.
	In particular, $M_n$ converge to $M$ in measure $\mu$ with respect to the product topology on $\R^\infty$ (which is metrizable, so one can define the convergence in measure).
\end{theorem}

Note that to obtain the strong convergence in $\ell^2$ we need to have some uniform (in $n$) bounds on the tails $\sum_{i = j}^\infty \lambda_i^+(T_n)$ and $\sum_{i = j}^\infty \lambda_i^+(T)$. It is not clear how can we get it, even if the operator $T$ is trace-class.

Another way to obtain a strong convergence is to modify slightly the definition of the infinite MDS as for instance proposed in the following statement. 

\begin{proposition}
	Suppose that $\diam(X) < \infty$, $p \ge 4$,
	and set 
	\begin{align*}
		\tilde{M}(x) &\eqdef \left(\sqrt{\lambda_j^+(T)} \frac{u_j^+(T)(x)}{\norm{u_j^+(T)}_{L^{p}(X,\mu)}}\right)_{j \ge 1},\\
	    \tilde{M}_n(x) &\eqdef \left(\sqrt{\lambda_j^+(T_n)} \frac{u_j^+(T_n)(x)}{\norm{u_j^+(T_n)}_{L^{p}(X,\mu)}}\right)_{j=1,\ldots, n},
	\end{align*}
i.e.\ the maps of the same form as $M$ and $M_n$ respectively, but with the eigenfunctions normalized in $L^{p}(X,\mu)$ instead of $L^2(X,\mu)$.  
	One has then $\tilde{M} \in L^{p}(X, \mu; \ell^{p})$. Moreover, there are orthogonal matrices $Q_n$ (with different dimensions) such that
	\begin{equation}\label{eq_4conv1d}
		\norm{\tilde{M} - Q_n \tilde{M}_n}_{L^{p}(X,\mu; \ell^p)} \to 0
	\end{equation}
	as $n \to \infty$. Here $Q_n x$ acts on the first $\dim Q_n$ components of $x \in \ell^p$, leaving all the other components unchanged.
\end{proposition}

\begin{remark}
	The condition on boundedness of $X$ (i.e.\ $\diam(X) < \infty$) is only taken to ensure that
	$\norm{u_j^+(T)}_{L^p(X,\mu)} <+\infty$ for all $j$. Of course it can be weakened to just an appropriate integrability of
	the kernel $\kap_T$ (or, equivalently, $\kap$) with respect to $\mu$.
\end{remark}

\begin{proof}
If $\diam(X) < \infty$, then by Lemma~\ref{lm_MDSLip} every $u_j^+(T)$ 
is bounded (by $C/\lambda_j$ for some $C>0$ depending only on $\diam X$)
so that $\norm{u_j^+(T)}_{L^{p}(X,\mu)} < +\infty$ and hence in view of Proposition~\ref{prop_MDScorrect1} the map $\tilde M$ is correctly defined.
We assume $u_j^+(T)$ to be already normalized in $L^p(X,\mu)$, so that
\[
\tilde{M}(x) \eqdef \left(\sqrt{\lambda_j^+(T)} u_j^+(T)(x)\right)_{j \ge 1}.  
\]
Denoting by $\langle \cdot,\cdot\rangle$ the duality pairing between $\ell^p$ and $\ell^{p'}$, 
for every $s \in (\ell^p)' = \ell^{p'}$, $p' = \frac{p}{p-1}$, 
one has
\[
\langle \tilde M(x), s\rangle = \sum_j s_j \sqrt{\lambda_j^+(T)} u_j^+(T)(x),
\] 
one has
\[
|\langle \tilde M(x), s\rangle| \leq \norm{s}_{\ell^{p'}} \sum_j (\lambda_j^+(T))^{p/2} |u_j^+(T)|^p(x) .
\]
Let us set $q \eqdef \frac{p}{2} \ge 2$ and denote by $\norm{A}_{S_q}$ the $q$-Schatten norm of a self-adjoint operator $A$, i.e.
\[
\norm{A}_{S_q} \eqdef \left(\sum_j |\lambda_j(A)|^{q}\right)^{1/q} .
\]
Then by  \cite[theorem 1]{russo1977hausdorff}, since $T$ is self-adjoint, one has
\begin{align*}
	\norm{T}_{S_q} &\le \left(\int_X \norm{\kap_T(\cdot, x)}_{L^{q'}(X, \mu)}^q d \mu(x)\right)^{1/q} \\
	&\le 4 \left(\int_X \norm{\kap(\cdot, x)}_{L^{q'}(X, \mu)}^q d \mu(x)\right)^{1/q}
	\le 2 \diam^2(X) .
\end{align*}
Hence, the series
\[
\sum_j (\lambda_j^+(T))^{p/2} |u_j^+(T)|^p(x)
\]
is convergent for $\mu$-a.e.\ $x\in X$, because
\[
\int_X \sum_j (\lambda_j^+(T))^{p/2} |u_j^+(T)|^p(x) = \sum_j (\lambda_j^+(T))^{p/2} < +\infty.
\]
Thus
\[
\langle \tilde M(x), s\rangle = \lim_m \sum_{j= 1}^m s_j \sqrt{\lambda_j^+(T)} u_j^+(T)(x)
\]
for $\mu$-a.e.\ $x\in X$, and hence the function $s\in (\ell^p)' \mapsto \langle \tilde M(x), s\rangle \in \R$ is measurable as a $\mu$-a.e.\ limit
of a sequence of continuous functions, implying that $\tilde M \colon X \to \ell^p$ is weakly (hence also strongly) measurable.
The equality
\[
\int_X \norm{\tilde M(x)}_{\ell^p}^p\, d\mu(x) = \sum_j (\lambda_j^+(T))^{p/2} \le \norm{T}_{S_{p/2}}^{p/2} 
<\infty
\]
shows that
$\tilde M \in L^p(X, \mu; \ell^p)$.

Using again \cite[theorem 1]{russo1977hausdorff}, the Lyapunov inequality (or, equivalently, just H\"{o}lder inequality), 
and the fact that $q' \le 2 \le q$ we obtain that
\begin{align*}
	\norm{T_n - T}_{S_q} &\le \left(\int_X \norm{\kap_T(\cdot, x) - \kap_{T_n}(\cdot, x)}_{L^{q'}(X, \mu)}^q d \mu(x)\right)^{1/q} \\
	&\le  \norm{\kap_T - \kap_{T_n}}_{L^q(X \times X, \mu \otimes \mu)} .
\end{align*}
Since 
\begin{align*}
	|\kap(x, y) - \kap_n(x, y)| &= \frac{\dist(x, y) + \dist(t_n(x), t_n(y))}{2} |\dist(x, y) - \dist(t_n(x), t_n(y))| \\
	&\le \diam(X) \left[\dist(x, t_n(x)) + \dist(y, t_n(y))\right] ,
\end{align*}
it is easy to see that 
\begin{align*}
	 \norm{\kap_T - \kap_{T_n}}_{L^q(X \times X, \mu \otimes \mu)}
	&\le 4 \norm{\kap - \kap_n}_{L^q(X \times X, \mu \otimes \mu)} \\
	&\le 8 \diam(X) \left( \int_X \dist(x, t_n(x))^q d \mu(x) \right)^{1/q} \\
	&\le 8 \diam(X)^{2 - 2 / q} W_2^{2 / q}(\mu, \mu_n) \to 0.
\end{align*}

Recall that if $v_n \in L^\infty(X, \mu)$ are uniformly bounded and $v_n \to v$ in $L^2(X, \mu)$, then they also converge in $L^p(X, \mu)$.
Thus in a similar way to Theorem~\ref{theo:W_stability}, using the uniform bounds on $u_j^+(T)$, we get that
\begin{equation}\label{eq_4conv1c}
	\lim_n \min_{Q \in O(m)} \norm{\tilde{M}^m - Q \tilde{M}^m_n}_{L^p(X,\mu; \ell^p)} = 0
\end{equation}
for every $m \in \N$ such that $\lambda^+_m(T) > \lambda^+_{m+1}(T)$.
Now note that 
\begin{equation}\label{eq_4conv1a}
	\norm{\tilde{M}^m - \tilde{M}}_{L^p(X,\mu; \ell^p)}^p 
	= \sum_{j = m+1}^\infty \left(\lambda_j^+(T)\right)^{p/2} \to 0 
\end{equation}
as $m \to \infty$, and extending a $Q\in O(m)$ to an operator (still denoted by $Q$) over $\ell^p$, we get
\begin{align*}\label{eq_4conv1b}
	\lim_{n \to \infty} \norm{Q\tilde{M}^m_n - Q\tilde{M}_n}_{L^p(X,\mu; \ell^p)}^p & =
	\lim_{n \to \infty} \norm{\tilde{M}^m_n - \tilde{M}_n}_{L^p(X,\mu; \ell^p)}^p \\
	&= \lim_{n \to \infty} \sum_{j = m+1}^\infty \left(\lambda_j^+(T_n)\right)^{p/2}\\
	&= \sum_{j = m+1}^\infty \left(\lambda_j^+(T)\right)^{p/2}
\end{align*}
by Proposition~\ref{prop:spectrum_stab}, since $\norm{T_n - T}_{S_{p/2}} \to 0$.
Hence,
\begin{equation}\label{eq_4conv1b1}
	\lim_{n \to \infty} \norm{Q \tilde{M}^m_n - Q \tilde{M}_n}_{L^p(X,\mu; \ell^p)} \to 0
\end{equation}
as $m \to \infty$.
Combining~\eqref{eq_4conv1c} with~\eqref{eq_4conv1a} and~\eqref{eq_4conv1b1}, we get~\eqref{eq_4conv1d},
concluding the proof.
\end{proof}

\section{MDS for sample spaces}\label{sec:mds-spheres}

In this section we show by means of several examples that MDS for manifolds equipped with intrinsic distances might produce a snowflake embedding of such manifolds in an infinite-dimensional Hilbert space, as, for instance, in the case of finite-dimensional spheres or flat tori.

\subsection{MDS embedding of spheres}\label{sec:sphere}
Consider the $d$-dimensional sphere $\mathbb{S}^d$ endowed with the inner metric $\dist(x, y) = \arccos(x\cdot y)$ and
a normalized surface measure $\mu \eqdef \frac{\sigma}{\sigma(\mathbb{S}^d)}$. 
Along with this space we consider its ``snowflake'', i.e.\ the same sphere but with the distance $\dist^{1/2}$ and the same
measure. We let $M$ and $M_{1/2}$ stand for the respective infinite MDS embeddings of $\mathbb{S}^d$ into $\ell^2$.
The following statement is valid.

\begin{proposition}\label{prop_MDSsphere1}
	One has
	\begin{equation}\label{eq_MDSsphere1}
		\norm{M(x) - M(y)}_2^2 = \pi \norm{M_{1/2}(x) - M_{1/2}(y)}_2^2 = \pi \dist(x, y).
	\end{equation}	
\end{proposition}
 
\begin{proof}
Define the kernel 
\[
\kap(x, y) = - \frac{1}{2} \dist^2(x, y) =- \frac{1}{2} \arccos^2(x\cdot y )
\]
over $\mathbb{S}^d\times \mathbb{S}^d$,
the corresponding integral operator $K$ on $L^2(X,\mu)$
\begin{equation}\label{eq:K_operator}
K[f](x) = \int_{\mathbb{S}^d} \kap(x, y) f(y) \,d \mu(y)
\end{equation}
and 
$T = P K P$,
where $P$ is the orthogonal projector onto the orthogonal complement of constant functions $\mathrm{span}^\bot(\one)$ in $L^2(\mathbb{S}^d,\mu)$ defined by the formula 
$$ P f := f - \one (\one, f)_{L^2(\mathbb{S}^d, \mu)}.$$
We justify below (see \eqref{eq:identity-a-b}) that the kernel $\kap$ enjoys the explicit series expansion
\begin{equation}\label{eq:a-series}
\begin{split}
\kap(x, y) = \sum_{n = 0}^\infty a_n ( x\cdot y)^n,\quad\text{where } 
a_n = \begin{cases}
- \frac{\pi^2}{8}, & n = 0, \\
\frac{\pi (2 j)!}{(2 j + 1) 2^{2 j + 1} (j!)^2}, & n = 2 j + 1, \\
- \frac{4^j (j!)^2}{2 (j + 1) (2 j + 1)}, & n = 2 j + 2.
\end{cases}
\end{split}
\end{equation}
Then by the theorem from~\cite[section~2]{azevedo2015eigenvalues} any spherical harmonic of degree $j$ is an eigenfunction of $K$ with eigenvalue 
\begin{equation}\label{eq:eigen-K}
\lambda_j(K) = \frac{\Gamma(d/2)}{2^{j+1}} \sum_{s = 0}^\infty a_{2 s + j} \frac{(2 s + j)!}{(2 s)!} \cdot \frac{\Gamma(s + 1 / 2)}{\Gamma(s + j + (d + 1) / 2)}.
\end{equation}
In particular, $\lambda_0(K)$ corresponds to a constant function, thus any other eigenfunction $u_j$ of $K$ has zero mean.
Since $\mathbb{S}^d$ is homogeneous under the action of the orthogonal group $O(d+1)$, and $\mu$ is invariant under this action, then by Proposition~\ref{prop_MDShomogeneous}
the operators $T$ and $K$ share the same set of eigenfunctions and eigenvalues except for $\lambda_0(T) = 0$.
We are interested only in positive eigenvalues, i.e.\ with odd indices. It is worth noting that according to~\eqref{eq:lambda_bound0} one has
\[
\sum_n \lambda_n^+(T)= \sum_n \lambda_n^+(K) <+\infty, 
\]
and hence by Proposition~\ref{prop_MDScorrect1} one has $M\in L^2(\mathbb{S}^d,\mu;\ell^2)$.

For the snowflake of $\mathbb{S}^{d}$ (endowed with the distance $\dist^{1/2}$) we define the kernel
\[
\kap_{1/2}(x, y) \eqdef - \frac{1}{2} \arccos( x \cdot y)
\]
over $\mathbb{S}^d\times \mathbb{S}^d$
and operators
\begin{equation}
K_{1/2}[f](x) \eqdef \int_{\mathbb{S}^d} \kap_{1/2}(x, y) f(y) \,d \mu(y), \quad T_{1/2} \eqdef P K_{1/2} P.
\end{equation}
in $L^2(X,\mu)$.
{ The Taylor series expansion of $\arccos$ function 
	gives
\begin{equation*}
\begin{aligned}
\kap_{1/2}(x, y) &= \sum_{n = 0}^\infty b_n (x\cdot y)^n, \quad\text{where }\\
b_n & = \begin{cases}
- \frac{\pi}{4}, & n = 0, \\
\frac{(2 j)!}{(2 j + 1) 2^{2 j + 1} (j!)^2}, & n = 2 j + 1, \\
0, & \text{otherwise}.
\end{cases}
\end{aligned}
\end{equation*}
The formula \eqref{eq:a-series} for $a_n$ can be 
either obtained from Taylor series expansion of the square of $\arccos$ function, or just
 from the above expression for $b_n$ through the simple calculation
\begin{align*}
\sum_{n=0}^\infty (-2a_n) (-1)^n x^n &= \arccos^2 (-x)=(\pi-\arccos x)^2 \\
&= \pi^2-2\pi \arccos x + \arccos^2 x \\
& =\pi^2-2\pi \sum_{n=0}^\infty (-2 b_n) x^n + \sum_{n=0}^\infty (-2a_n) x^n,
\end{align*} 
which implies
\begin{equation}\label{eq:identity-a-b}
\sum_{n=0}^\infty (-2a_n) ((-1)^n-1) x^n = \pi^2-2\pi \sum_{n=0}^\infty (-2b_n) x^n
\end{equation}
for all $x\in \R$. }
In particular, for $n\neq 0$ odd one has $a_n=\pi b_n$, while for $n\neq 0$ even one has $b_n=0$. 
In short,  for $n \ge 1$ we have $\pi b_n = (a_n)^+$.
Using again results from~\cite{azevedo2015eigenvalues} we immediately get that any spherical harmonic of degree $j$ is an eigenfunction of $K_{1/2}$ with the eigenvalue 
\[
\lambda_j(K_{1/2}) = \frac{\Gamma(d/2)}{2^{j+1}} \sum_{s = 0}^\infty b_{2 s + j} \frac{(2 s + j)!}{(2 s)!} \cdot \frac{\Gamma(s + 1 / 2)}{\Gamma(s + j + (d + 1) / 2)}.
\]
The latter clearly is positive for $j\neq 0$, if and only if $j$ is odd, in which case $\pi \lambda_j(K_{1/2}) = \lambda_j(K)$.
By the same reason as the operators $T$ and $K$, also the operators $T_{1/2}$ and $K_{1/2}$ share the same set of eigenfunctions and eigenvalues  except for $\lambda_0(T_{1/2}) = 0$.
Hence
\begin{equation}\label{eq:T_relation}
T_+ = \pi T_{1/2},
\end{equation}
where $T_+$ is the ``positive part'' of $T$, i.e.\
\[
T_+ v := \sum_{j} \lambda_j^+(T) (v, u_j^+(T))_{L^2(\mathbb{S}^d, \mu)}  u_j^+(T).
\]
Therefore, $T_{1/2}$ is positive semidefinite and trace-class, 
hence an MDS embedding of $(\mathbb{S}^{d}, \dist^{1/2})$ into $\ell^2$ is an isometric in view of Remark~\ref{rm_positiveMDS}. Hence the second equality of~\eqref{eq_MDSsphere1} holds. The first equality of~\eqref{eq_MDSsphere1} is true in view of~\eqref{eq:T_relation}.
\end{proof}

\subsection{MDS of product spaces}\label{sec:product}
Let $(X_i, \dist_i, \mu_i)$ be a metric space with probability measure $\mu_i$, $i = 1, 2$. 
Then define a metric on $X \eqdef X_1 \times X_2$ by
\[
\dist(x, y) := \sqrt{\dist_1^2(x_1, y_1) + \dist_2^2(x_2, y_2)}, \quad \text{where } x = (x_1, x_2),\; y = (y_1, y_2),
\]
and endow this space with the product measure $\mu := \mu_1 \otimes \mu_2$.

\begin{proposition}\label{prop_MDSprod1}
	Denote by $M^{(i)}\colon X^{(i)}\to \ell^2$, the infinite MDS map for each space  $(X_i, \dist_i, \mu_i)$, $i = 1, 2$. 
	Then the  infinite MDS map $M$ for the space $(X,\dist,\mu)$ can be represented as
	\begin{align*}
	M \colon& X \to \ell^2\times \ell^2, \\
	&x \mapsto M(x)=(M^{(1)}(x_1), M^{(2)}(x_2)),
	\end{align*}
	the space $\ell^2\times \ell^2$ being equipped with the natural Hilbert space structure,
	and
	\[
	\norm{M(x) - M(y)}_2^2 = \norm{M^{(1)}(x_1) - M^{(1)}(y_1)}_2^2 + \norm{M^{(2)}(x_2) - M^{(2)}(y_2)}_2^2.
	\]		
\end{proposition}

\begin{proof}
	Let $\kap^{(i)}$, $\kap_T^{(i)}$, $K^{(i)}$, and $T^{(i)}$ denote the corresponding kernels and operators on $X_i$. Consider the kernel 
	\[\kap(x, y) \eqdef - \frac{1}{2} \dist^2(x, y) = \kap^{(1)}(x_1, y_1) + \kap^{(2)}(x_2, y_2)\] 
	over $X$ and the respective operators $K$ and $T$. 
Note that
\begin{align*}
\int_X \kap(x, y') \,d \mu(y') &= \int_{X_1 \times X_2} \left(\kap^{(1)}(x_1, y_1') + \kap^{(1)}(x_2, y_2')\right) \,d \mu_1(y_1') \otimes \mu_2(y_2') \\
&= \int_{X_1} \kap^{(1)}(x_1, y_1') \,d \mu_1(y_1') + \int_{X_2} \kap^{(2)}(x_2, y_2') \,d \mu_2(y_2'),
\end{align*}
thus the kernel of $T$
\begin{align*}
\kap_T(x, y) &= \kap(x, y) - \int_X \kap(x, y') \,d \mu(y') - \int_X \kap(x', y) \,d \mu(x') \\
&\qquad + \int_X \int_X \kap(x', y') \,d \mu(x') \,d \mu(y') \\
&= \kap_T^{(1)}(x_1, y_1) + \kap_T^{(2)}(x_2, y_2).
\end{align*}

To study the eigenfunctions of $T$ recall that $\{u_i \otimes v_j\}_{i, j}$ is an orthonormal basis in $L^2(X, \mu)$, where $u_i$ is an eigenfunction of $T^{(1)}$ with eigenvalue $\nu_i$, $v_j$ is an eigenfunction of $T^{(2)}$ with eigenvalue $\eta_j$, and $u_i \otimes v_j \eqdef u_i(x_1) v_j(x_2)$. Then
\begin{align*}
\left(T u_i \otimes v_j\right)(x) &= \int_{X} \left(\kap_T^{(1)}(x_1, y_1) + \kap_T^{(2)}(x_2, y_2)\right) u_i(y_1) v_j(y_2) \,d \mu_1(y_1) \otimes \mu_2(y_2) \\
&= \int_{X_1} \kap_T^{(1)}(x_1, y_1) u_i(y_1) \,d \mu_1(y_1) \int_{X_2} v_j(y_2) \,d \mu_2(y_2) \\
&~~~+ \int_{X_1} u_i(y_1) \,d \mu_1(y_1) \int_{X_2} \kap_T^{(2)}(x_2, y_2) v_j(y_2) \,d \mu_2(y_2) \\
&= \bigl(T^{(1)} u_i\bigr)(x_1) \cdot ( v_j, \one)_{L^2(X_2,\mu_2)} + \bigl(T^{(2)} v_j\bigr)(x_2) \cdot  (u_i, \one)_{L^2(X_1,\mu_1)}   \\
&= \nu_i u_i(x_1) (v_j, \one)_{L^2(X_2,\mu_2)}   + \eta_j v_j(x_2) (u_i, \one)_{L^2(X_1,\mu_1)}  .
\end{align*}
Note that constant functions belong to $\ker T^{(1)}$ and $\ker T^{(2)}$, thus other eigenfunctions have zero mean. Hence, $T (u_i \otimes v_j) \neq 0$ only if $v_j = \one$, $\nu_i \neq 0$, or $u_i = \one$, $\eta_j \neq 0$. In the former case we have
\[
T (u_i \otimes v_j) = \nu_i u_i \otimes \one = \nu_i (u_i \otimes v_j),
\]
and in the latter case
\[
T (u_i \otimes v_j) = \eta_j \one \otimes v_j = \eta_j (u_i \otimes v_j).
\]
Therefore, $\{u_i \otimes v_j\}_{i, j}$ is a basis of eigenfunctions of $T$, with only eigenfunctions outside $\ker T$ having the form $\{u_i \otimes \one\}_i$ or $\{\one \otimes v_j\}_j$. We immediately obtain that one can represent the MDS map $M$ of $(X, \dist, \mu)$ as claimed.
\end{proof}

An immediate corollary for the case of the flat torus $\mathbb{S}^1\times \mathbb{S}^1$ is as follows.

\begin{proposition}\label{prop_MDStor1}
The infinite MDS map $M\colon \mathbb{S}^1\times \mathbb{S}^1\to \ell^2$ for the flat torus $\mathbb{S}^1\times \mathbb{S}^1$ equipped with the locally Euclidean (flat) metric and the 
volume measure satisfies
	\[
	\norm{M((x_1, x_2)) - M((y_1, y_2))}_2^2 = \pi\left(\dist_{\mathbb{S}^1}(x_1, y_1)+ \dist_{\mathbb{S}^1}(x_2, y_2)
	\right),
	\]		
	where $\dist_{\mathbb{S}^1}$ stands for the spherical distance (i.e.\ the geodesic distance),
	i.e.\ provides a bi-H\"{o}lder embedding of the flat torus in a Hilbert space.
\end{proposition}

\begin{proof}
	Combine Proposition~\ref{prop_MDSprod1} with Proposition~\ref{prop_MDSsphere1} (with $d \eqdef 1$).
\end{proof}

Of course, a similar result can be easily obtained for a flat torus of any dimension.

\appendix

\section{Estimates of eigenvalues on a sphere}
We show that for the positive eigenvalues of the operator $K$ defined by~\eqref{eq:K_operator}, i.e. those with odd indices, one has 
\begin{equation}\label{eq:lambda_bound0}
\lambda_{2 n + 1}(K)= \Theta_d\left(n^{- d - 1}\right),
\end{equation}
where we use Knuth's big Theta notation: $f(\cdot) = \Theta(g(\cdot))$ means that $f(\cdot) = O(g(\cdot))$ and $g(\cdot) = O(f(\cdot))$ hold simultaneously. Here and below the subscript $d$ means that we consider $d$ fixed, so the constant in $O(\cdot)$ can depend on it, and all the asymptotic notation, unless otherwise stated, refers to $n\to \infty$.
To this aim note that by \eqref{eq:eigen-K} one has
\begin{equation}\label{eq:positive_lambda}
\lambda_n^+(K) 
= \lambda_{2 n + 1}(K) 
= \Gamma(d/2) \sum_{s = 0}^\infty \theta_n(s),
\end{equation}
where
\begin{align*}
\theta_{n}(s)
& = \frac{1}{2^{2 n + 2}} a_{2 s + 2 n + 1} \frac{(2 s + 2 n + 1)!}{(2 s)!} \cdot \frac{\Gamma(s + 1 / 2)}{\Gamma(s + 2 n + (d + 3) / 2)} \\
& = \frac{1}{2^{2 n + 2}} \cdot \frac{\pi (2 (s + n))!}{(2 (s + n) + 1) 2^{2 (s + n) + 1} ((s + n)!)^2} \cdot \frac{(2 (s + n) + 1)!}{(2 s)!} \cdot \\
&~~~\cdot \frac{\Gamma(s + 1 / 2)}{\Gamma(s + 2 n + (d + 3) / 2)} 
\quad 
\qquad  \text{(writing $a_{{2s+2n+1}}$ from \eqref{eq:a-series})} \\
& = 2^{2 s - 3} \left(\frac{(2 (s + n))! \sqrt{\pi}}{2^{2 (s + n)} (s + n)!}\right)^2 \frac{1}{(2 s)!} \cdot \frac{(2 s)! \sqrt{\pi}}{2^{2 s} s!} \cdot \frac{1}{\Gamma(s + 2 n + (d + 3) / 2)} \\
& = \frac{\sqrt{\pi}}{8} \frac{\Gamma^2(s + n + 1 / 2)}{\Gamma(s + 2 n + (d + 3) / 2) s!}.
\end{align*}
Therefore,
\begin{align*}
\alpha_n(s) &\eqdef \frac{\theta_{n}(s+1)}{\theta_{n}(s)} 
= \frac{(s + n + 1 / 2)^2}{(s + 1) (s + 2 n + (d + 3) / 2)} \\
&= 1 - \frac{2 (d + 3) (s + 1) - (2 n - 1)^2}{2 (s + 1) (2 s + 4 n + d + 3)} \\
&= 1 - \frac{1}{2 s + 4 n + d + 3}\left((d + 3) - \frac{(2 n - 1)^2}{2 (s + 1)}\right).
\end{align*}
From the above relationship we obtain that
for
\[
s^*:=\frac{(2 n - 1)^2}{2 (d + 3)} - 1 
\]
one has that $\alpha_n(s^*)=1$, $\alpha_n(s)\leq 1$ when $s\geq s^*$, and $\alpha_n(s)\geq 1$ when $s\leq s^*$.
Thus we get
\begin{equation}\label{eq:s_n}
s_n \eqdef \mathrm{argmax}_{s \in \N} \theta_n(s)=\left\lceil s^*\right\rceil = \left\lceil\frac{(2 n - 1)^2}{2 (d + 3)} - 1\right\rceil = \Theta_d(n^2).
\end{equation}
Note that
\begin{equation}\label{eq:suggested}
\sum_{s = s_n}^\infty \theta_n(s) \le \sum_{s = 0}^\infty \theta_n(s) \le s_n \theta_n(s_n) + \sum_{s = s_n}^\infty \theta_n(s).
\end{equation}

Stirling's formula yields 
\[
\Gamma(x) = \Theta\left(\frac{1}{\sqrt{x}} \left(\frac{x}{e}\right)^x\right),
\]
for large $x>0$, and 
thus
\begin{align*}
\theta_{n}(s) 
& = \Theta\left(\frac{\left(\frac{s + n - 1 / 2}{e}\right)^{2 (s + n)}}{\left(\frac{s + 2 n + (d + 1) / 2}{e}\right)^{s + 2 n + d / 2 + 1} \left(\frac{s}{e}\right)^{s + 1 / 2}}\right) \\
& = \Theta\left(\frac{\left(s + n - 1 / 2\right)^{2 s + 2 n}}{\left(s + 2 n + (d + 1) / 2\right)^{s + 2 n + d / 2 + 1} s^{s + 1 / 2}}\right) \\
& = \Theta\left(\left(\frac{e}{s}\right)^{(d + 3) / 2} \left(1 + \frac{n - 1 / 2}{s}\right)^{2 s + 2 n} \left(1 + \frac{2 n + (d + 1) / 2}{s}\right)^{- (s + 2 n + d / 2 + 1)}\right).
\end{align*}
Since $\ln(1 + x) = x - O(x^2)$ as $x\to 0$, we obtain that 
\[
(2 s + 2 n) \ln\left(1 + \frac{n - 1 / 2}{s}\right) = 2 n - 1 + O\left(\frac{n^2}{s} + \frac{n^3}{s^2}\right),
\]
as well as
\[
\left(s + 2 n + \frac{d}{2} + 1\right) \ln\left(1 + \frac{2 n + (d + 1) / 2}{s}\right)
= 2 n + \frac{d + 1}{2} + O\left(\frac{(n + d)^2}{s} + \frac{(n + d)^3}{s^2}\right).
\]
Hence for any $s \ge s_n = \Theta_d(n^2)$ one has
\[
\ln \theta_n(s) = - \frac{d + 3}{2} (\ln s - 1) - 1 - \frac{d + 1}{2} + O_d\left(\frac{n^2}{s}\right) = - \frac{d + 3}{2} \ln s + O_d(1),
\]
and thus $\theta_{n}(s) = \Theta_d\left(s^{- (d + 3) / 2}\right)$.
Therefore,
\[
\sum_{s = s_n}^\infty \theta_n(s) = \Theta_d\left(\sum_{s = s_n}^\infty s^{- (d + 3) / 2}\right)
= \Theta_d\left(s_n^{- (d + 1) / 2}\right).
\]
Moreover, 
\[
s_n \theta_n(s_n) = \Theta_d\left(s_n^{- (d + 1) / 2}\right).
\]
Finally, substituting the latter estimate into~\eqref{eq:suggested}, from~\eqref{eq:positive_lambda} we conclude that
\begin{equation}\label{eq:lambda_bound}
\lambda_{2 n + 1}(K) = \Theta_d\left(\sum_{s = 0}^\infty \theta_n(s)\right)
= \Theta_d\left(s_n^{- (d + 1) / 2}\right) = \Theta_d\left(n^{- d - 1}\right),
\end{equation}
where in the latter equility we used again $s_n=\Theta_d(n^2)$. This is in fact~\eqref{eq:lambda_bound0} as claimed.
\bibliographystyle{plain}

\end{document}